\documentclass[11pt,abstracton]{scrartcl}

\usepackage{amsmath,amssymb,amsthm,tikz}
\usepackage[english]{babel}

\usepackage{hyperref}
\hypersetup{
  pdftitle     = Stability and Hermitian-Einstein metrics for vector bundles on framed manifolds,
  pdfauthor    = Matthias Stemmler,
  pageanchor   = false,
  pdfborder    = 0 0 0,
  pdfpagemode  = UseNone,
  pdfstartview = FitH,
  plainpages   = true
}

\usetikzlibrary{matrix}

\areaset{16cm}{22.2cm}
\addtokomafont{sectioning}{\rmfamily}


\newtheorem{theorem}{Theorem}
\newtheorem{corollary}[theorem]{Corollary}
\newtheorem{lemma}[theorem]{Lemma}
\newtheorem{proposition}[theorem]{Proposition}

\newtheorem*{theorem*}{Theorem}

\theoremstyle{definition}
\newtheorem{definition}[theorem]{Definition}

\theoremstyle{remark}
\newtheorem{remark}[theorem]{Remark}

\newcommand{\dbar}{\bar \partial}
\newcommand{\eps}{\varepsilon}
\newcommand{\stack}[2]{\genfrac{}{}{0pt}{}{#1}{#2}} 
\newcommand{\tensor}{\otimes}

\newcommand{\CC}{\mathbb{C}}
\newcommand{\RR}{\mathbb{R}}

\newcommand{\C}{\mathcal{C}}
\newcommand{\E}{\mathcal{E}}
\newcommand{\F}{\mathcal{F}}

\renewcommand{\O}{\mathcal{O}}
\renewcommand{\P}{\mathcal{P}}

\DeclareMathOperator{\End}{End}
\DeclareMathOperator{\id}{id}
\DeclareMathOperator{\rank}{rank}
\DeclareMathOperator{\Ric}{Ric}
\DeclareMathOperator{\tr}{tr}
\DeclareMathOperator{\vol}{vol}

\renewcommand{\Re}{\operatorname{Re}}

\renewcommand{\leq}{\leqslant}
\renewcommand{\geq}{\geqslant}

\title{Stability and Hermitian-Einstein metrics for vector bundles on framed manifolds}
\author{Matthias Stemmler}

\begin{document}

\maketitle

\begin{abstract}
We adapt the notions of stability of holomorphic vector bundles in the sense of Mumford-Takemoto and Hermitian-Einstein metrics in holomorphic vector bundles for canonically polarized framed manifolds, i.~e.\ compact complex manifolds $X$ together with a smooth divisor $D$ such that $K_X \tensor [D]$ is ample. It turns out that the degree of a torsion-free coherent sheaf on $X$ with respect to the polarization $K_X \tensor [D]$ coincides with the degree with respect to the complete K\"ahler-Einstein metric $g_{X \setminus D}$ on $X \setminus D$. For stable holomorphic vector bundles, we prove the existence of a Hermitian-Einstein metric with respect to $g_{X \setminus D}$ and also the uniqueness in an adapted sense.
\end{abstract}


%
%
\section{Introduction}

The Kobayashi-Hitchin correspondence states that an irreducible holomorphic vector bundle on a compact K\"ahler manifold admits a Hermitian-Einstein metric if and only if it is stable in the sense of Mumford-Takemoto. A {\em logarithmic pair\/} (or {\em framed manifold}) is defined to be a pair $(X, D)$ consisting of a compact complex manifold $X$ together with a smooth (irreducible) divisor $D$ in $X$. It is called {\em canonically polarized\/} if the line bundle $K_X \tensor [D]$ is ample, where $K_X$ denotes the canonical line bundle of $X$ and $[D]$ is the line bundle associated to the divisor~$D$. In this case, corresponding to Yau's solution of the Calabi conjecture \cite{Yau78}, R.~Kobayashi \cite{Ko84} proves the existence and uniqueness (up to a constant multiple) of a complete K\"ahler-Einstein metric $g_{X'}$ with negative Ricci curvature on the complement $X' := X \setminus D$ of $D$ in $X$. This metric is of Poincar\'e-type growth near the divisor $D$. The manifold $X'$ has bounded geometry in the sense of Cheng and Yau \cite{CY80} with respect to quasi-coordinates that were introduced by R.~Kobayashi in \cite{Ko84} and studied by Tian and Yau in \cite{TY87}. Moreover, the asymptotic behaviour of the metric $g_{X'}$ near the divisor $D$ is known due to results by Schumacher \cite{Sch98}, \cite{Sch02}. The goal of the present article is to use these results to establish suitable notions of stability and Hermitian-Einstein metrics for vector bundles on framed manifolds in order to obtain a Kobayashi-Hitchin correspondence in this context.


One approach involves {\em parabolic structures\/} in a holomorphic vector bundle on~$X$ with respect to the divisor $D$ as introduced by Mehta and Seshadri \cite{MS80} on Riemann surfaces and generalized to higher-dimensional varieties by Maruyama and Yokogawa \cite{MY92} (see also Biswas \cite{Bs95}, \cite{Bs97a}, \cite{Bs97b}). Li and Narasimhan \cite{LN99} studied parabolic structures for rank-$2$ vector bundles on framed manifolds of complex dimension $2$ with respect to a K\"ahler metric on~$X'$ which is the restriction of a K\"ahler metric on~$X$. Moreover, Biquard \cite{Bi97} studied these structures with respect to the complete K\"ahler-Einstein metric on $X'$.

We follow a different approach by directly adapting the ordinary notions of stability and Hermitian-Einstein metrics to the framed situation, making use of the results on the asymptotic behaviour of the Poincar\'e-type metric. Given a canonically polarized framed manifold $(X, D)$, there are two obvious notions of ``stability in the framed sense'' of a torsion-free coherent analytic sheaf~$\E$ on the framed manifold $(X, D)$. On the one hand, there is the standard notion of stability of $\E$ with respect to the polarization $K_X \tensor [D]$ of~$X$. On the other hand, we can compute the degree of coherent subsheaves $\F$ of $\E$ in terms of the complete K\"ahler-Einstein metric on $X'$. For this notion of stability it is sufficient to consider subsheaves which are defined on $X$ rather than $X'$. Using a theorem from \cite{Sch98} on the asymptotic behaviour of the Poincar\'e-type metric, we show that these two approaches are equivalent. This is strong evidence that the resulting notion of {\em stability in the framed sense\/} or {\em framed stability\/} is suitable for a Kobayashi-Hitchin correspondence. In fact, as in the classical situation, the framed stability of a holomorphic vector bundle on~$X$ implies its simplicity. However, it seems not to imply the simplicity of its restriction to $X'$.

The precise definition of a {\em Hermitian-Einstein metric in the framed sense\/} or {\em framed Hermitian-Einstein metric\/} in a holomorphic vector bundle $E$ on~$X$ requires special attention. We are interested in Hermitian metrics in $E$ over $X'$ satisfying the Hermitian-Einstein condition with respect to the complete K\"ahler-Einstein metric. Like in Simpson's work \cite{Si88} and in the parabolic case, we also need a compatibility condition for Hermitian-Einstein metrics in order to exclude unwanted connected components in the space of metrics.

Having established these notions, we prove the following version of the Kobayashi-Hitchin correspondence: In every holomorphic vector bundle $E$ on a canonically polarized framed manifold which is stable in the framed sense, there exists a unique (up to a constant multiple) framed Hermitian-Einstein metric. Thanks to the concept of bounded geometry, Simpson's heat equation method from \cite{Si88} can be applied to our situation by expressing all analytic considerations in terms of quasi-coordinates. Simpson solves an evolution equation of the heat conduction type for all non-negative values of a time parameter $t$. If the solution converges as $t$ goes to infinity, the limit yields the desired Hermitian-Einstein metric. For the case of divergence he constructs a subsheaf contradicting the stability condition. This is done by first constructing a {\em weakly holomorphic subbundle\/} of $E$ given by a measurable section $\pi$ of $\End(E)$ in the Sobolev space of $L^2$ sections with $L^2$ first-order weak derivatives. In our situation, the section $\pi$ satisfies the $L^2$ condition with respect to the Poincar\'e-type metric. The estimates from \cite{Sch02} imply the $L^2$ condition in the usual sense, so that the theorem of Uhlenbeck-Yau \cite{UY86} can be applied, which provides a destabilizing subsheaf of $\E$ on all of $X$.

We also want to point out the following result by Ni and Ren (\cite{NR01}, see also \cite{Ni02}): In the situation of a complete K\"ahler manifold with a positive lower bound on the spectrum of the Laplace-Beltrami operator, they show that every Hermitian metric which is ``asymptotically Hermitian-Einstein'' (i.~e.\ with its contracted curvature tensor satisfying a certain integrability condition) can be deformed into a Hermitian-Einstein metric. It turns out, however, that manifolds with metrics of Poincar\'e-type growth always have finite volume (see \cite{Ko84}, proof of Lemma 1, p.\ 402) and thus clearly violate the above spectral bound condition. In particular, the theorem of Ni and Ren does not cover the situation considered here. Other results in this direction include the works of Simpson \cite{Si88} and Bando~\cite{Ba93}.


The results in this paper are from the author's dissertation \cite{St09}.

%
%
\section{Preliminaries}

Let $(X,g)$ be a compact K\"ahler manifold of complex dimension $n$. The notion of $g$-stability in the sense of Mumford-Takemoto \cite{Ta72} can be formulated as follows. Define the {\em g-degree\/} of a torsion-free coherent analytic sheaf~$\E$ on~$X$ as
\[
  \deg_g(\E) = \int_X c_1(\E) \wedge \omega^{n-1},
\]
where $c_1(\E)$ is the first Chern class of $\E$ and $\omega$ is the fundamental form of $g$. If $\E = \O_X(E)$ is the sheaf of holomorphic sections of a holomorphic vector bundle $E$ on~$X$, the $g$-degree can be written as
\[
  \deg_g(E) = \int_X \frac{\sqrt{-1}}{2 \pi} \tr(F_h) \wedge \omega^{n-1},
\]
where $F_h$ is the curvature form of the Chern connection in $E$ with respect to a smooth Hermitian metric~$h$. If the sheaf~$\E$ is non-trivial, the ratio
\[
  \mu_g(\E) = \frac{\deg_g(\E)}{\rank(\E)}
\]
of the $g$-degree of $\E$ and its rank is called the {\em $g$-slope\/} of $\E$. A torsion-free coherent analytic sheaf~$\E$ on~$X$ is then called {\em $g$-stable\/} if
\[
  \mu_g(\F) < \mu_g(\E)
\]
holds for every coherent subsheaf~$\F$ of $\E$ with $0 < \rank(\F) < \rank(\E)$. A holomorphic vector bundle $E$ on~$X$ is called {\em $g$-stable\/} if its sheaf~$\E = \O_X(E)$ of holomorphic sections is $g$-stable. In the projective-algebraic case, given an ample line bundle $H$ on~$X$, the degree of $\E$, which is then called the {\em $H$-degree}, is an integer and can be written as
\[
  \deg_H(\E) = (c_1(\E) \cup c_1(H)^{n-1}) \cap [X].
\]
Correspondingly, in this case, the notions of slope and stability of a torsion-free coherent analytic sheaf~$\E$ or a holomorphic vector bundle $E$ on~$X$ are referred to as the the {\em $H$-slope\/} and the {\em $H$-stability}, respectively. According to S.~Kobayashi \cite{Kb80}, a Hermitian metric~$h$ in a holomorphic vector bundle $E$ on~$X$ is called a {\em $g$-Hermitian-Einstein metric\/} if it satisfies the Einstein condition
\[
  \sqrt{-1} \Lambda_g F_h = \lambda_h \id_E
\]
with a real constant $\lambda_h$, where $\sqrt{-1} \Lambda_g$ is the contraction with $\omega$, $F_h$ is the curvature form of the Chern connection of $(E,h)$ and $\id_E$ is the identity endomorphism of $E$. The relationship of the notions of $g$-stability and $g$-Hermitian-Einstein metrics is classical: An irreducible holomorphic vector bundle on~$X$ admits a $g$-Hermitian-Einstein metric if and only if it is $g$-stable, cf.\ Narasimhan/Seshadri \cite{NS65}, S.~Kobayashi \cite{Kb82}, L\"ubke \cite{Lue83}, Donaldson \cite{Do83}, \cite{Do85}, \cite{Do87} and Uhlenbeck/Yau \cite{UY86}, \cite{UY89}.

The definition of framed manifolds given in the introduction is analogous to {\em framed vector bundles\/} and related notions, cf.\ Lehn \cite{Le93}, L\"ubke \cite{Lue93}, L\"ubke/Okonek/Schumacher \cite{LOS93}, Garc\'ia-Prada \cite{GP94}, Bradlow/Garc\'ia-Prada \cite{BG96}, Schmitt \cite{Sc04} and Bruasse/Teleman \cite{BT05}. In the canonically polarized case, the complete K\"ahler-Einstein metric on $X'$ mentioned above will be denoted by~$g_{X'}$ in what follows. It is of {\em Poincar\'e-type growth\/} near the divisor~$D$, i.~e.\ for every point $p \in D$, there is a coordinate neighbourhood $U(p) \subset X$ of $p$ with $U(p) \cap X' \cong \Delta^\ast \times \Delta^{n-1}$ such that in these coordinates, it is quasi-isometric to a product of the Poincar\'e metric on the punctured unit disc~$\Delta^\ast$ and $n-1$ copies of the Euclidean metric on the full unit disc~$\Delta$. As mentioned in the introduction, $X'$ has finite volume with respect to this metric.

Given a local coordinate system $(\Delta^\ast \times \Delta^{n-1}; z^1, \ldots, z^n)$ on an open neighbourhood $U(p) \subset X$ of a point $p \in D$ as above with respect to which the divisor~$D$ is given by the equation $z^1 = 0$, the quasi-coordinates mentioned above are defined as follows. Choose a fixed real number~$R$ with $\frac{1}{2} < R < 1$ and write $B_R(0) = \{ v \in \CC : |v| < R \}$. Then there are quasi-coordinate systems $(B_R(0) \times \Delta^{n-1}; v^1, \ldots, v^n)$ with
\[
  v^1 = \frac{w^1 - a}{1 - a w^1}, \quad \text{where }
  z^1 = \exp\left(\frac{w^1 + 1}{w^1 - 1}\right),
\]
and
\[
  v^i = w^i = z^i \quad \text{for } 2 \leq i \leq n,
\]
where $a$ varies over all real numbers in $\Delta$ close to $1$. There is a representation of $dv^1$ in terms of $dz^1$ which does not directly involve the numbers $a$. In fact, we have
\begin{equation} \label{e:quasi-differentials}
  dv^1 = \frac{(|v^1|^2 - 1) (v^1 - 1)}{\bar v^1 - 1} \, \frac{dz^1}{z^1 \log(1/|z^1|^2)}.
\end{equation}
When doing analysis with respect to the complete K\"ahler-Einstein metric, the function spaces defined in terms of these quasi-coordinates provide the most convenient choice in order to carry over the arguments from the compact (un-framed) case.

The asymptotic behaviour of this metric is known --- in fact, we have the following explicit description of its volume form.

\begin{theorem*}[Schumacher, \cite{Sch98}, Theorem 2]
There is a number $0 < \alpha \leq 1$ such that for all $k \in \{ 0, 1, \ldots \}$ and $\beta \in (0, 1)$, the volume form of $g_{X'}$ is of the form
\[
  \frac{2 \Omega}{||\sigma||^2 \log^2 (1/||\sigma||^2)}
  \left(1 + \frac{\nu}{\log^\alpha (1/||\sigma||^2)}\right)
  \quad \text{with } \nu \in \C^{k,\beta}(X'),
\]
where $\Omega$ is a smooth volume form on~$X$, $\sigma$ is a canonical section of~$[D]$, $||{\cdot}||$ is the norm induced by a Hermitian metric in $[D]$ and $\C^{k,\beta}(X')$ is the H\"older space of $\C^{k,\beta}$ functions on~$X'$ with respect to the quasi-coordinates.
\end{theorem*}
Moreover, in \cite{Sch98}, it is shown that the fundamental form of $g_{X'}$ converges to the fundamental form of a K\"ahler-Einstein metric on $D$ locally uniformly when restricted to coordinate directions parallel to $D$. From this, one obtains the following result on the asymptotics of the Poincar\'e-type metric. Let $\sigma$ be a canonical section of~$[D]$, which can be regarded as a local coordinate in a neighbourhood of a point $p \in D$. Then we can choose local coordinates $(\sigma, z^2, \ldots, z^n)$ near $p$ such that if $g_{\sigma \bar \sigma}$, $g_{\sigma \bar \jmath}$ etc.\ denote the coefficients of the fundamental form of $g_{X'}$ and $g^{\bar \sigma \sigma}$ etc.\ denote the entries of the corresponding inverse matrix, we have the following statement from~\cite{Sch02}.

\begin{proposition} \label{p:asymptotics-preliminaries}
With $0 < \alpha \leq 1$ from the above theorem, we have
\begin{enumerate}
\item[(i)]   $g^{\bar \sigma \sigma} \sim |\sigma|^2 \log^2 (1/|\sigma|^2)$,
\item[(ii)]  $g^{\bar \sigma i}, g^{\bar \jmath \sigma} = O\big(|\sigma| \log^{1-\alpha} (1/|\sigma|^2)\big)$, $i, j = 2, \ldots, n$,
\item[(iii)] $g^{\bar \imath i} \sim 1$, $i = 2, \ldots, n$ and
\item[(iv)]  $g^{\bar \jmath i} \to 0$ as $\sigma \to 0$, $i, j = 2, \ldots, n$, $i \neq j$,
\end{enumerate}
where $a \sim b$ denotes the existence of a constant $c > 0$ such that $\frac{1}{c} \, a \leq b \leq ca$.
\end{proposition}

%
%
\section{Framed stability}

Let $(X,D)$ be a canonically polarized framed manifold, i.~e.\ $X$ is a compact complex manifold and $D$ is a smooth divisor in $X$ such that $K_X \tensor [D]$ is ample. As above, we write $X' := X \setminus D$ for the complement of $D$ in $X$. When looking for a good notion of stability of torsion-free coherent analytic sheaves on~$X$ with respect to the framed manifold $(X,D)$, the critical aspect is the definition of the degree of such sheaves. The following two notions seem reasonable.

\begin{itemize}
\item Since $H := K_X \tensor [D]$ is an ample line bundle on~$X$, there is the notion of $(K_X \tensor [D])$-stability. In this case, we define the degree of a torsion-free coherent analytic sheaf~$\E$ on~$X$ as
\[
  \deg_H(\E) = (c_1(\E) \cup c_1(H)^{n-1}) \cap [X].
\]
This means that the degree is computed with respect to a K\"ahler metric on~$X$ whose fundamental form is obtained from the curvature form of a positive Hermitian metric in the line bundle $K_X \tensor [D]$.

\item By Kobayashi's theorem, there is a unique (up to a constant multiple) complete K\"ahler-Einstein metric $g_{X'}$ on~$X'$ with negative Ricci curvature and Poincar\'e-type growth near the divisor $D$. We can thus define the degree of a torsion-free coherent analytic sheaf~$\E$ on~$X$ as
\[
  \deg_{X'}(\E) = \int_{X'} c_1(\E) \wedge \omega_{X'}^{n-1},
\]
where $\omega_{X'}$ is the fundamental form of $g_{X'}$. When following this approach, we have to make sure that the integral is well-defined and, in particular, independent of the choice of the closed real $(1,1)$-form representing $c_1(\E)$.
\end{itemize}

The aim of this section is to show that these two ways of computing the degree of a torsion-free coherent analytic sheaf on~$X$ are equivalent and so there is only one notion of ``framed stability'' of such sheaves. Note that while the first approach is a special case of stability in the ordinary (un-framed) sense on~$X$ (namely, with respect to a special polarization), the second approach is not a special case of stability in the ordinary sense on~$X'$ because here one only considers subsheaves on~$X$ instead of $X'$.

In order to show the well-definedness of $\deg_{X'}(\E)$, we need the following lemma.

\begin{lemma} \label{l:integrable}
If $\eta$ is a smooth real $(1,1)$-form on~$X$, we have
\[
  \int_{X'} |\Lambda_g \eta| dV_g < \infty,
\]
where $g = g_{X'}$ is the complete K\"ahler-Einstein metric on~$X'$ with volume form $dV_g$ and $\Lambda_g$ is the formal adjoint of forming the $\wedge$-product with the fundamental form $\omega_{X'}$ of $g_{X'}$.
\end{lemma}

\begin{proof}
Using local coordinates $z^1, \ldots, z^n$ on an open neighbourhood $U \subset X$ of a point $p \in D$ and writing
\begin{align*}
  \omega_{X'} &= \sqrt{-1} g_{i \bar \jmath} dz^i \wedge dz^{\bar \jmath}, \\
  \eta        &= \eta_{i \bar \jmath} dz^i \wedge dz^{\bar \jmath}
\end{align*}
with smooth local functions $\eta_{i \bar \jmath}$, $i, j = 1, \ldots, n$, we have
\[
  \sqrt{-1} \Lambda_g \eta = g^{\bar \jmath i} \eta_{i \bar \jmath}
\]
and thus
\[
  |\Lambda_g \eta|^2 = g^{\bar \jmath i} \eta_{i \bar \jmath} g^{\bar \ell k} \eta_{k \bar \ell}.
\]
If, in particular, $(z^1, z^2, \ldots, z^n) = (\sigma, z^2, \ldots, z^n)$ is the coordinate system of Proposition~\ref{p:asymptotics-preliminaries}, $g^{\bar \jmath i}$ is bounded for $i, j = 1, \ldots, n$. Since the $\eta_{i \bar \jmath}$ are smooth functions, we obtain that $|\Lambda_g \eta|$ is bounded. The claim then follows by the finite volume of $(X', g_{X'})$.
\end{proof}

Furthermore, we need the following generalization of Stokes' theorem for complete Riemannian manifolds due to Gaffney.

\begin{theorem*}[Gaffney, \cite{Ga54}]
Let $(M, ds_M^2)$ be an orientable complete Riemannian manifold of real dimension $2n$ whose Riemann tensor is of class $\C^2$. Let $\gamma$ be a $(2n-1)$-form on $M$ of class $\C^1$ such that both $\gamma$ and $d\gamma$ are in $L^1$. Then
\[
  \int_M d\gamma = 0.
\]
\end{theorem*}

\begin{lemma} \label{l:degree on complement}
If $\E$ is a torsion-free coherent analytic sheaf on~$X$, the integral
\begin{equation} \label{e:degree on complement}
  \deg_{X'}(\E) = \int_{X'} c_1(\E) \wedge \omega_{X'}^{n-1}
\end{equation}
is well-defined and, in particular, independent of the choice of a closed real $(1,1)$-form $c_1(\E)$ representing the first Chern class of $\E$.
\end{lemma}

\begin{proof}
Let $\eta$ be a closed smooth real $(1,1)$-form on~$X$ representing $c_1(\E)$. Then we have
\[
  \eta \wedge \omega_{X'}^{n-1} = (n-1)! (\Lambda_g \eta) \frac{\omega_{X'}^n}{n!}
\]
and Lemma \ref{l:integrable} implies the existence of the integral \eqref{e:degree on complement}.

Now if $\tilde \eta$ is another such $(1,1)$-form representing $c_1(\E)$, we have $\eta - \tilde \eta = d \zeta$ for a smooth $1$-form $\zeta$ on~$X$. It follows that
\[
  \int_{X'} \eta \wedge \omega_{X'}^{n-1} - \int_{X'} \tilde \eta \wedge \omega_{X'}^{n-1}
  = \int_{X'} d\zeta \wedge \omega_{X'}^{n-1}
  = \int_{X'} d\gamma,
\]
where $\gamma := \zeta \wedge \omega_{X'}^{n-1}$ is a smooth $(2n-1)$-form on~$X'$ such that $d\gamma$ and (as can be shown analogously) $\gamma$ itself are in $L^1$. Now apply Gaffney's theorem with $(M, ds_M^2)$ being the underlying Riemannian manifold of $(X', g_{X'})$, which is complete by Kobayashi's theorem. This implies $\int_{X'} d\gamma = 0$, thus proving the claim.
\end{proof}

We can now prove the equivalence of the two notions of degree discussed above.

\begin{proposition} \label{p:equivalence of degrees}
Let $\E$ be a torsion-free coherent analytic sheaf on~$X$. Then
\[
  \deg_H(\E) = \deg_{X'}(\E),
\]
where $H := K_X \tensor [D]$.
\end{proposition}

\begin{proof}
Let $\eta$ be a closed smooth real $(1,1)$-form on~$X$ representing $c_1(\E)$. Then we have
\[
  \deg_H(\E) = \int_X \eta \wedge \omega_X^{n-1},
\]
where $\omega_X$ is obtained from the curvature form of a positive Hermitian metric in $H = K_X \tensor [D]$, i.~e.
\[
  \omega_X
  = - \Ric\left(\frac{\Omega}{||\sigma||^2}\right)
  = \sqrt{-1} \partial \dbar \log\left(\frac{\Omega}{||\sigma||^2}\right)
\]
with a smooth volume form $\Omega$ on~$X$ and a smooth Hermitian metric~$h$ in $[D]$ with induced norm~$||{\cdot}||$ such that $\omega_X$ is positive definite. Here, as above, $\sigma$ denotes a canonical holomorphic section of~$[D]$. On the other hand, we have
\[
  \deg_{X'}(\E) = \int_{X'} \eta \wedge \omega_{X'}^{n-1},
\]
where $\omega_{X'}$ is the fundamental form of the metric $g_{X'}$ on~$X'$. By Schumacher's theorem on the asymptotics of this metric and the fact that it is K\"ahler-Einstein, there is a number $0 < \alpha \leq 1$ such that (in particular) for all $k \geq 2$ and $\beta \in (0,1)$, we have
\[
  \omega_{X'}
  = - \Ric(\omega_{X'}^n)
  = \sqrt{-1} \partial \dbar \log \left(\frac{2 \Omega}{||\sigma||^2 \log^2(1/||\sigma||^2)}
    \left(1 + \frac{\nu}{\log^\alpha (1/||\sigma||^2)}\right)\right)
\]
with a function $\nu \in \C^{k,\beta}(X')$. A comparison of $\omega_X$ and $\omega_{X'}$ yields
\[
  \begin{split}
    \omega_{X'} & = \sqrt{-1} \partial \dbar \log \left(\frac{2 \Omega}{||\sigma||^2 \log^2 (1/||\sigma||^2)}
                    \left(1 + \frac{\nu}{\log^\alpha (1/||\sigma||^2)}\right)\right) \\
                & = \sqrt{-1} \partial \dbar \log \left(\frac{\Omega}{||\sigma||^2}\right)
                    - 2 \sqrt{-1} \partial \dbar \log \log (1/||\sigma||^2)
                    + \sqrt{-1} \partial \dbar \log \left(1 + \frac{\nu}{\log^\alpha (1/||\sigma||^2)}\right)
  \end{split}
\]
and thus
\begin{equation} \label{e:comparing fundamental forms}
\omega_{X'} = \omega_X|_{X'}
              - 2 \sqrt{-1} \partial \dbar \log \log (1/||\sigma||^2)
              + \sqrt{-1} \partial \dbar \log \left(1 + \frac{\nu}{\log^\alpha (1/||\sigma||^2)}\right).
\end{equation}
For notational convenience, we first do the proof for the case of $n=2$ and then explain the necessary changes for the proof to work in higher dimensions as well.

Since $X' = \bigcup_{\eps > 0} X_\eps$ with $X_\eps = \{ x \in X : ||\sigma(x)|| > \eps \}$, we have
\[
  \deg_H(\E) = \lim_{\eps \to 0} \int_{X_\eps} \eta \wedge \omega_X
  \quad \text{and} \quad
  \deg_{X'}(\E) = \lim_{\eps \to 0} \int_{X_\eps} \eta \wedge \omega_{X'}
\]
and, therefore,
\[
  \begin{split}
  \deg_{X'}(\E) & =     \deg_H(\E) - 2 \sqrt{-1} \lim_{\eps \to 0}
                        \int_{X_\eps} \eta \wedge \partial \dbar \log \log (1/||\sigma||^2) \\
                & \quad + \sqrt{-1} \lim_{\eps \to 0}
                        \int_{X_\eps} \eta \wedge \partial \dbar \log \left(1 + \frac{\nu}{\log^\alpha (1/||\sigma||^2)}\right) \\
                & =     \deg_H(\E) + 2 \sqrt{-1} \lim_{\eps \to 0}
                        \int_{X_\eps} d\left(\eta \wedge \partial \log \log (1/||\sigma||^2)\right) \\
                & \quad - \sqrt{-1} \lim_{\eps \to 0}
                        \int_{X_\eps} d\left(\eta \wedge \partial \log \left(1 + \frac{\nu}{\log^\alpha (1/||\sigma||^2)}\right)\right) \\
                & =     \deg_H(\E) + 2 \sqrt{-1} \lim_{\eps \to 0}
                        \int_{\partial X_\eps} \eta \wedge \partial \log \log (1/||\sigma||^2) \\
                & \quad - \sqrt{-1} \lim_{\eps \to 0}
                        \int_{\partial X_\eps} \eta \wedge \partial \log \left(1 + \frac{\nu}{\log^\alpha (1/||\sigma||^2)}\right)
  \end{split}
\]
by Stokes' theorem. It remains to show that
\begin{align}
  \label{e:vanishing integral 1}
  \lim_{\eps \to 0} \int_{\partial X_\eps} \eta \wedge \partial \log \log (1/||\sigma||^2) &= 0, \\
  \label{e:vanishing integral 2}
  \lim_{\eps \to 0} \int_{\partial X_\eps} \eta \wedge \partial \log \left(1 + \frac{\nu}{\log^\alpha (1/||\sigma||^2)}\right) &= 0.
\end{align}
We have $\partial X_\eps = \{ x \in X : ||\sigma(x)|| = \eps \}$. By abuse of notation, we regard $\sigma$ as a local coordinate on an open neighbourhood $U \subset X$ of a point $p \in D$ and regard $h$ as a smooth positive function on $U$. Then we have local coordinates $(\sigma, z)$ on $U$ such that $||\sigma||^2 = |\sigma|^2/h$. In \eqref{e:vanishing integral 1}, we have
\[
  \partial \log \log (1/||\sigma||^2)
  = \frac{\partial \log (1/||\sigma||^2)}{\log(1/||\sigma||^2)}
  = \frac{\partial \log h - \partial \log |\sigma|^2}{\log(1/||\sigma||^2)}
  = \frac{\partial \log h - \frac{d\sigma}{\sigma}}{\log(1/\eps^2)}
  \quad \text{on } \partial X_\eps,
\]
and thus
\[
  \int_{\partial X_\eps} \eta \wedge \partial \log \log (1/||\sigma||^2)
  = \frac{1}{\log(1/\eps^2)} \left(
    \int_{\partial X_\eps} \eta \wedge \partial \log h
    - \int_{\partial X_\eps} \eta \wedge \frac{d\sigma}{\sigma}
    \right).
\]
The first integral is clearly bounded uniformly in $\eps$. The second integral can be estimated as follows. By Fubini's theorem, it suffices to estimate a one-dimensional line integral of the form
\[
  \int_{||\sigma|| = \eps} \frac{f(\sigma) d\sigma}{\sigma},
\]
where $f$ is a smooth locally defined function involving the coefficients of $\eta$. Since by the $\C^1$ version of Cauchy's integral formula (see, e.~g., H\"ormander \cite{Hoe90}, Theorem 1.2.1), we have
\[
  f(0) = \frac{1}{2 \pi \sqrt{-1}} \int_{||\sigma|| = \eps} \frac{f(\sigma) d\sigma}{\sigma}
         + \frac{1}{2 \pi \sqrt{-1}} \iint_{||\sigma|| < \eps} \frac{\partial f}{\partial \bar \sigma} \frac{d\sigma \wedge d\bar \sigma}{\sigma}
\]
and $f(0)$ is a finite number, it suffices to estimate the area integral. The latter is, however, bounded uniformly in $\eps$ since $f$ is smooth and, writing $\sigma = r e^{i \varphi}$ in polar coordinates, we have
\[
  \left|\frac{d\sigma \wedge d\bar \sigma}{\sigma}\right|
  = \left|\frac{- 2 \sqrt{-1} r dr \wedge d\varphi}{r e^{i \varphi}}\right|
  = 2 |dr \wedge d\varphi|.
\]
As $\log(1/\eps^2) \to \infty$ as $\eps \to 0$, we obtain \eqref{e:vanishing integral 1}. In \eqref{e:vanishing integral 2}, we have
\[
  \partial \log\left(1 + \frac{\nu}{\log^\alpha(1/||\sigma||^2)}\right)
  = \frac{1}{1 + \frac{\nu}{\log^\alpha(1/\eps^2)}}
    \left(\frac{\partial \nu}{\log^\alpha(1/\eps^2)}
    - \frac{\alpha \nu \left(\partial \log h - \frac{d \sigma}{\sigma}\right)}{\log^{\alpha+1}(1/\eps^2)}
    \right)
  \quad \text{on } \partial X_\eps,
\]
and thus
\[
  \begin{split}
    \int_{\partial X_\eps} \eta \wedge \partial \log\left(1 + \frac{\nu}{\log^\alpha(1/||\sigma||^2)}\right)
    & =     \frac{1}{\log^\alpha(1/\eps^2)} \int_{\partial X_\eps} \frac{\eta \wedge \partial \nu}{1 + \frac{\nu}{\log^\alpha(1/\eps^2)}} \\
    & \quad - \frac{\alpha}{\log^{\alpha+1}(1/\eps^2)} \int_{\partial X_\eps} \frac{\eta \wedge \nu \partial \log h}
            {1 + \frac{\nu}{\log^\alpha(1/\eps^2)}} \\
    & \quad + \frac{\alpha}{\log^{\alpha+1}(1/\eps^2)} \int_{\partial X_\eps} \frac{\eta \wedge \nu \frac{d\sigma}{\sigma}}
            {1 + \frac{\nu}{\log^\alpha(1/\eps^2)}}. \\
  \end{split}
\]
Again, by Fubini's theorem, it suffices to consider the one-dimensional situation. Since $\nu$ is in $\C^{k,\beta}(X')$ with $k \geq 2$, $\nu$ is (in particular) bounded on~$X'$ and so
\[
  \sup_{\partial X_\eps} \left|\frac{1}{1 + \frac{\nu}{\log^\alpha(1/\eps^2)}}\right|
\]
is bounded uniformly in $\eps$ and so is the second integral above. Moreover, if $v$ is the quasi-coordinate corresponding to $\sigma$, we have
\[
  \partial \nu
  = \frac{\partial \nu}{\partial v} dv
  = \frac{\partial \nu}{\partial v} \frac{(|v|^2 - 1)(v - 1)}{(\bar v - 1) \log(1/|\sigma|^2)} \frac{d\sigma}{\sigma}
\]
by \eqref{e:quasi-differentials}, where $\frac{\partial \nu}{\partial v}$ is bounded on~$X'$. Consequently, the other two integrals can be bounded by using Cauchy's integral formula as above. Since
\[
  \log^\alpha(1/\eps^2) \rightarrow \infty
  \quad \text{ and } \quad
  \log^{\alpha+1}(1/\eps^2) \rightarrow \infty
  \quad \text{ as } \eps \rightarrow 0,
\]
we obtain \eqref{e:vanishing integral 2}. This concludes the proof for the case of $n=2$.

In dimension $n>2$, one expands the expression $\omega_{X'}^{n-1}$, where $\omega_{X'} = \omega_X|_{X'} + \xi$ is written as in \eqref{e:comparing fundamental forms} with a closed smooth real $(1,1)$-form $\xi$ on~$X'$. Then one has to show the vanishing for $\eps \to 0$ of several integrals of the forms \eqref{e:vanishing integral 1} and \eqref{e:vanishing integral 2} with additional terms which are either equal to $\omega_X$ or to $\xi$. Since $\omega_X$ is smooth on~$X$, it does not destroy the convergence. Concerning $\xi$, an argument similar to the one in the proof of Lemma \ref{l:degree on complement} shows that this does not influence the convergence either. Thus the proof works in any dimension.
\end{proof}

We can now proceed in parallel to the compact case.

\begin{definition}[Framed degree, framed slope]
Let $\E$ be a torsion-free coherent analytic sheaf on~$X$.
\begin{enumerate}
\item[(i)] We call the integer
\[
  \deg_{(X,D)}(\E) := \deg_H(\E) = \deg_{X'}(\E)
\]
from Proposition~\ref{p:equivalence of degrees} the {\em framed degree\/} or the {\em degree in the framed sense\/} of $\E$ with respect to the framed manifold $(X,D)$.
\item[(ii)] If $\rank(\E) > 0$, we call
\[
  \mu_{(X,D)}(\E) := \frac{\deg_{(X,D)}(\E)}{\rank(\E)}
\]
the {\em framed slope\/} or the {\em slope in the framed sense\/} of $\E$ with respect to the framed manifold $(X,D)$.
\end{enumerate}
\end{definition}

\begin{definition}[Framed stability]
A torsion-free coherent analytic sheaf~$\E$ on~$X$ is said to be {\em stable in the framed sense\/} with respect to the framed manifold $(X,D)$ if the inequality
\[
  \mu_{(X,D)}(\F) < \mu_{(X,D)}(\E)
\]
holds for every coherent subsheaf~$\F$ of $\E$ with $0 < \rank(\F) < \rank(\E)$.
\end{definition}

A holomorphic vector bundle $E$ on~$X$ is said to be {\em stable in the framed sense\/} if its sheaf $\E = \O_X(E)$ of holomorphic sections has this property. Since the framed stability of $E$ with respect to $(X,D)$ is a special case of the stability of $E$ in the ordinary sense on~$X$ (namely, it is the stability with respect to the polarization $K_X \tensor [D]$), we have the following corollary from the classical situation.

\begin{corollary} \label{c:framed stable-simple}
If $E$ is a stable holomorphic vector bundle on~$X$ in the framed sense with respect to $(X,D)$, then $E$ is {\em simple}, i.~e.\ every holomorphic section of its bundle of endomorphisms is a constant multiple of the identity endomorphism.
\end{corollary}

Note, however, that the framed stability of $E$ with respect to $(X,D)$ seems not to imply the simplicity of $E' := E|_{X'}$. Thus, given a holomorphic section of $\End(E)$ over $X'$, one has to make sure that it can be holomorphically extended to the whole of~$X$ in order to conclude that it is a scalar multiple of the identity.

%
%
\section{Framed Hermitian-Einstein metrics}

Regarding a suitable notion of Hermitian-Einstein metrics in the framed sense in a holomorphic vector bundle $E$ on $X$, our interest lies on smooth Hermitian-Einstein metrics in the restriction $E'$ of $E$ to $X'$ with respect to the complete K\"ahler-Einstein metric. In order to ensure that everything will be well-defined in the following considerations, we first have to make a restriction on the class of smooth Hermitian metrics in $E'$ given by Simpson in \cite{Si88}: Denote by $\P$ the space of smooth Hermitian metrics $h'$ in $E'$ such that
\[
  \int_{X'} |\Lambda_g F_{h'}|_{h'} dV_g < \infty,
\]
where $F_{h'}$ is the curvature form of the Chern connection of the Hermitian holomorphic vector bundle $(E',h')$ on~$X'$. First of all, if $h'$ is the restriction to $E'$ of a smooth Hermitian metric~$h$ in $E$, we have $h' \in \P$ by Lemma \ref{l:integrable}. Now the definition of $\P$ is such that for any $h' \in \P$, the integral
\[
  \deg_{X'}(E',h') := \int_{X'} \frac{\sqrt{-1}}{2 \pi} \tr(F_{h'}) \wedge \omega_{X'}^{n-1}
\]
is well-defined. However, in order to ensure that it equals the framed degree $\deg_{(X,D)}(E)$ of $E$ with respect to $(X,D)$, we have to impose an additional condition on $h'$. Following Simpson \cite{Si88}, we denote by $S_{h'}$ the bundle of endomorphisms of $E'$ which are self-adjoint with respect to $h'$. Furthermore, we let $P(S_{h'})$ be the space of smooth sections $s$ of $S_{h'}$ such that
\[
  ||s||_P := \sup_{X'} |s|_{h'} + ||\nabla'' s||_{L^2} + ||\Delta' s||_{L^1} < \infty,
\]
where $\nabla = \nabla' + \nabla''$ is the covariant derivative on smooth sections of $\End(E')$ with respect to the Chern connection of the Hermitian holomorphic vector bundle $(E',h')$ and $\Delta' = \sqrt{-1} \Lambda_g \nabla'' \nabla'$ is the $\nabla'$-Laplacian on smooth sections of $\End(E')$ with respect to $h'$ and the metric $g_{X'}$. Here, the $L^p$ norms are also defined with respect to $h'$ and $g_{X'}$. Now, according to \cite{Si88}, the space $\P$ can be turned into an analytic manifold with local charts
\[
  \begin{array}{rcl}
    P(S_{h'}) & \longrightarrow & \P      \\
    s         & \longmapsto     & h' e^s
  \end{array} .
\]
Divide $\P$ into maximal components such that each of these charts covers a component. Choose a smooth Hermitian metric~$h_0$ in $E$ and use the same notation $h_0$ for its restriction to $E'$. The component $\P_0$ of $\P$ containing $h_0$ is easily seen to be independent of the choice of $h_0$ because the restrictions to $E'$ of any two smooth Hermitian metrics in $E$ lie in the same component of $\P$. This space $\P_0$ turns out to be a suitable space in which to look for Hermitian-Einstein metrics with respect to the complete K\"ahler-Einstein metric.

\begin{definition}[Framed Hermitian-Einstein metric] \label{d:framed hermitian-einstein metric}
A smooth Hermitian metric~$h'$ in $E'$ is called a {\em framed Hermitian-Einstein metric\/} or {\em Hermitian-Einstein metric in the framed sense\/} in $E$ with respect to the framed manifold $(X,D)$ if $h' \in \P_0$ and
\[
  \sqrt{-1} \Lambda_g F_{h'} = \lambda_{h'} \id_{E'}
\]
with a constant $\lambda_{h'} \in \RR$, which is then called the {\em Einstein factor\/} of $h'$.
\end{definition}

\begin{lemma} \label{l:framed degree with metric}
If $h' \in \P_0$, we have
\[
  \deg_{X'}(E', h') = \deg_{(X,D)}(E).
\]
\end{lemma}

\begin{proof}
First of all, because of $h' \in \P_0 \subset \P$, the integral
\[
  \deg_{X'}(E',h')
  = \int_{X'} \frac{\sqrt{-1}}{2 \pi} \tr(F_{h'}) \wedge \omega_{X'}^{n-1}
  = \int_{X'} \frac{\sqrt{-1} (n-1)!}{2 \pi} \tr(\Lambda_g F_{h'}) \frac{\omega_{X'}^n}{n!}
\]
is well-defined. Furthermore, $\frac{\sqrt{-1}}{2 \pi} \tr(F_{h_0})$ is a closed real $(1,1)$-form on~$X$ representing the first Chern class $c_1(E)$ and thus
\[
  \deg_{(X,D)}(E)
  = \deg_{X'}(E)
  = \int_{X'} \frac{\sqrt{-1}}{2 \pi} \tr(F_{h_0}) \wedge \omega_{X'}^{n-1}
  = \int_{X'} \frac{\sqrt{-1} (n-1)!}{2 \pi} \tr(\Lambda_g F_{h_0}) \frac{\omega_{X'}^n}{n!}.
\]
We therefore have to show that
\begin{equation} \label{e:einstein factor claim}
  \int_{X'} (\tr(\Lambda_g F_{h'}) - \tr(\Lambda_g F_{h_0})) \frac{\omega_{X'}^n}{n!} = 0.
\end{equation}
Because of $h' \in \P_0$, we have $h' = h_0 e^s$ with $s \in P(S_{h_0})$. By Bott-Chern theory, we know that
\[
  \tr(\Lambda_g F_{h'}) - \tr(\Lambda_g F_{h_0}) = \Lambda_g \dbar \partial \tr(s).
\]
From $h', h_0 \in \P$, it follows that $\Lambda_g \dbar \partial \tr(s)$ is integrable on~$X'$. Also, because of $s \in P(S_{h_0})$, we know that $\dbar \tr(s) = \tr(\nabla'' s)$ is integrable on~$X'$. By Gaffney's theorem, \eqref{e:einstein factor claim} follows.
\end{proof}

\begin{remark} \label{r:framed einstein factor}
In particular, if $h'$ is a framed Hermitian-Einstein metric in $E$ with respect to $(X,D)$ and Einstein factor $\lambda_{h'}$, Lemma \ref{l:framed degree with metric} implies that
\[
  \lambda_{h'} = \frac{2 \pi \mu_{(X,D)}(E)}{(n-1)! \vol_g(X')}
\]
as in the classical theory, where $\vol_g(X') = \int_{X'} \frac{\omega_{X'}^n}{n!}$ is the {\em volume\/} of $X'$ with respect to $g_{X'}$.
\end{remark}

We can now show the uniqueness of a framed Hermitian-Einstein metric in a simple bundle up to a constant multiple. We give a detailed proof of this proposition in order to show how to carry over the arguments from the classical situation. The existence proof in the following section will then work in a similar way, except for the considerations on the $L^2$ condition mentioned in the introduction.

\begin{proposition} \label{p:uniqueness}
Let $E$ be a simple holomorphic vector bundle on a canonically polarized framed manifold $(X,D)$. Then if $h_0'$ and $h_1'$ are Hermitian-Einstein metrics in $E$ in the framed sense with respect to $(X,D)$, there is a constant $c > 0$ such that $h_1' = c \cdot h_0'$.
\end{proposition}

\begin{proof}
First of all, we have
\begin{equation} \label{e:hermitian-einstein-same-lambda}
  \sqrt{-1} \Lambda_g F_{h_0'} = \lambda \id_{E'} = \sqrt{-1} \Lambda_g F_{h_1'}
  \quad \text{with } \lambda = \frac{2 \pi \mu_{(X,D)}(E)}{(n - 1)! \vol_g(X')}
\end{equation}
by Remark \ref{r:framed einstein factor}. Since $h_0'$ and $h_1'$ lie in the same component $\P_0$ of $\P$, we know that $h_1' = h_0' e^s$ for some $s \in P(S_{h_0'})$. Join $h_0'$ and $h_1'$ by the path $h_t' = h_0' e^{ts}$ for $t \in [0,1]$ and define the function $L: [0,1] \to \CC$ by
\[
  L(t) = \int_{X'} \int_0^t \tr\big(s (\sqrt{-1} \Lambda_g F_{h_u'} - \lambda \id_{E'})\big) du \frac{\omega_{X'}^n}{n!}.
\]
This is a special version of Donaldson's functional. For the convenience of the reader, we give the first part of the now classical argument. The function $L$ is well-defined since for every $t \in [0,1]$, we have
\begin{align*}
         \left|\int_0^t \tr\big(s (\sqrt{-1} \Lambda_g F_{h_u'} - \lambda \id_{E'})\big) du\right|
  &\leq t \sup_{u \in [0,t]} \Big|\big< \sqrt{-1} \Lambda_g F_{h_u'} - \lambda \id_{E'}, s \big>_{h_u'}\Big| \\
  &=    t \Big|\big< \sqrt{-1} \Lambda_g F_{h_{u_0}'} - \lambda \id_{E'}, s \big>_{h_{u_0}'}\Big| \\
  &\leq t \big|\sqrt{-1} \Lambda_g F_{h_{u_0}'} - \lambda \id_{E'}\big|_{h_{u_0}'} |s|_{h_{u_0}'} \\
  &\leq t \left(|\Lambda_g F_{h_{u_0}'}|_{h_{u_0}'} + |\lambda| \sqrt{\rank(E)}\right) ||s||_P
\end{align*}
for some $u_0 \in [0,t]$, where the last expression is integrable over $X'$ with respect to $g_{X'}$ because of $h_{u_0}' \in \P_0 \subset \P$, $s \in P(S_{h_{u_0}'})$ and the finite volume of $(X', g_{X'})$. The first derivative of $L$ is
\[
  L'(t) = \int_{X'} \tr\big(s (\sqrt{-1} \Lambda_g F_{h_t'} - \lambda \id_{E'})\big) \frac{\omega_{X'}^n}{n!}
\]
and the Hermitian-Einstein condition \eqref{e:hermitian-einstein-same-lambda} yields $L'(0) = 0 = L'(1)$. By Bott-Chern theory, we know that
\[
  \frac{d}{dt} (\Lambda_g F_{h_t'}) = \Lambda_g \nabla'' \nabla_{h_t'}' s,
\]
where
\[
  \nabla_{h_t'} = \nabla_{h_t'}' + \nabla''
\]
is the covariant derivative on smooth sections of $\End(E')$ with respect to the Chern connection of the Hermitian holomorphic vector bundle $(E', h_t')$. Consequently, the second derivative of $L$ is
{\allowdisplaybreaks \begin{align*}
  L''(t) &= \int_{X'} \tr\big(s (\sqrt{-1} \Lambda_g \nabla'' \nabla_{h_t'}' s)\big) \frac{\omega_{X'}^n}{n!} \\
         &= \sqrt{-1} \int_{X'} \tr(s \nabla'' \nabla_{h_t'}' s) \wedge \frac{\omega_{X'}^{n-1}}{(n-1)!} \\
         &= - \sqrt{-1} \int_{X'} \tr(\nabla'' s \wedge \nabla_{h_t'}' s) \wedge \frac{\omega_{X'}^{n-1}}{(n-1)!}
            + \sqrt{-1} \int_{X'} \dbar \tr(s \nabla_{h_t'}' s) \wedge \frac{\omega_{X'}^{n-1}}{(n-1)!} \\
         &= ||\nabla'' s||_{L^2}^2 + \sqrt{-1} \int_{X'} d\gamma
\end{align*}}
since $s$ is self-adjoint with respect to $h_t'$, where the $L^2$ norm is as above and
\[
  \gamma = \tr(s \nabla_{h_t'}' s) \wedge \frac{\omega_{X'}^{n-1}}{(n-1)!}
\]
is a smooth $(2n-1)$-form on~$X'$. We are going to verify the hypotheses of Gaffney's theorem. We have
\[
  |\tr(s \nabla_{h_t'}' s)|
  \leq |\nabla_{h_t'}' s|_{h_t'} |s|_{h_t'}
  = |\nabla'' s|_{h_t'} |s|_{h_t'}
\]
and from $s \in P(S_{h_t'})$, we know that $|\nabla'' s|_{h_t'}$ is $L^2$ and $|s|_{h_t'}$ is bounded on~$X'$. It follows that $\gamma$ is $L^2$ on~$X'$ and, in particular, $L^1$ due to the finite volume of $(X', g_{X'})$. Moreover, we know that
\[
  |\Delta_{h_t'}' s|_{h_t'} = |\Lambda_g \nabla'' \nabla_{h_t'}' s|_{h_t'}
\]
is $L^1$ on~$X'$. Thus, $d\gamma$ is seen to be $L^1$ on~$X'$ as well. By Gaffney's theorem, it follows that $\int_{X'} d\gamma = 0$ and we obtain
\[
  L''(t) = ||\nabla'' s||_{L^2}^2
\]
for all $t \in [0,1]$. In particular, $L''(t)$ is independent of $t$. From $L'(0) = 0 = L'(1)$, it follows that $L' \equiv 0$ and thus $L'' \equiv 0$ on $[0,1]$. This implies that $\nabla'' s = 0$, i.~e.\ $s$ is a holomorphic section of $\End(E')$. As above, let $h_0$ be a smooth Hermitian metric in $E$. Then $h_0$ and $h_0'$ lie in the same component $\P_0$ of $\P$ and the boundedness of $|s|_{h_0'}$ implies the boundedness of $|s|_{h_0}$. By Riemann's extension theorem, $s$ can be extended to a holomorphic section of $\End(E)$ over $X$. Since the bundle $E$ is simple by hypothesis, we have $s = a \id_E$ for some number $a$, which must be real as $s$ is self-adjoint. Finally, we obtain
\[
  h_1' = h_0' e^s = c \cdot h_0'
  \quad \text{with } c = e^a > 0
\]
as claimed.
\end{proof}

%
%
\section{Square-integrability conditions}

Having established the notions of framed stability and framed Hermitian-Einstein metrics, we can now state a framed version of the Kobayashi-Hitchin correspondence.

\begin{theorem} \label{t:framed kobayashi-hitchin}
Let $E$ be a holomorphic vector bundle on a canonically polarized framed manifold $(X,D)$ such that $E$ is stable in the framed sense with respect to $(X,D)$. Then there is a unique (up to a constant multiple) Hermitian-Einstein metric in $E$ in the framed sense with respect to $(X,D)$.
\end{theorem}

The uniqueness follows from Corollary \ref{c:framed stable-simple} and Proposition~\ref{p:uniqueness}. As mentioned in the introduction, the classical existence proof (cf.\ \cite{Do85}, \cite{Si88}) can be applied to the framed situation almost completely using the notions of quasi-coordinates and bounded geometry together with Gaffney's theorem. The only critical aspect is the application of the regularity theorem for weakly holomorphic subbundles (cf.\ \cite{UY86}, \cite{UY89}, \cite{Po05}) to the framed situation. In order for this statement to produce a coherent subsheaf over the compact manifold $X$ as needed by the definition of framed stability, the weakly holomorphic subbundle must be given as an $L_1^2$ section with respect to a smooth K\"ahler metric on~$X$. Since it is first obtained as an $L_1^2$ section with respect to the Poincar\'e-type metric on $X'$, we can complete the proof by showing that the $L_1^2$ condition with respect to this metric implies the $L_1^2$ condition in the ordinary sense, which is the aim of this section.

First we define the relevant spaces of sections. Since $L^2$ sections are only defined almost everywhere and the divisor~$D$ has measure zero, there is no difference between considering $L^2$ sections on~$X$ and on~$X'$, so we introduce all notions on the compact manifold $X$. By the adjunction formula, we see that
\[
  K_D = (K_X \tensor [D])|_D
  \quad \text{is ample,}
\]
so that there is a unique (up to a constant multiple) K\"ahler-Einstein metric $g_D$ on $D$ with negative Ricci curvature. Let $\omega_D$ be its fundamental form and let $\sigma$ be a canonical section of~$[D]$. By abuse of notation, we regard $\sigma$ as a local coordinate function near a point $p \in D$. We can restrict $\omega_{X'}$ to the locally defined sets $D_{\sigma_0} := \{ \sigma = \sigma_0 \} \subset X'$ for small $\sigma_0 > 0$ and there is a notion of locally uniform convergence of $\omega_{X'}|D_{\sigma_0}$ for $\sigma_0 \to 0$. We then have the following convergence theorem.


\begin{theorem*}[Schumacher, \cite{Sch98}, Theorem 1]
$\omega_{X'}|D_{\sigma_0}$ converges to $\omega_D$ locally uniformly as $\sigma_0 \to 0$.
\end{theorem*}

In what follows, all estimates can be done in a small neighbourhood $U \subset X$ of an arbitrary point $p \in D$. This neighbourhood will be shrinked several times as needed throughout the computation. We can choose coordinates $z^2, \ldots, z^n$ for $D$ on $U \cap D$ such that
\[
  \omega_D = \sqrt{-1} \sum_{i=2}^n dz^i \wedge d\bar z^i
\]
is diagonal at $p$. Regarding $\sigma$ as a local coordinate, we have a coordinate system $(\sigma, z^2, \ldots, z^n)$ on $U$. We write $dV$ for the Euclidean volume element and $dV_g$ for the volume element of the metric $g = g_{X'}$. Then locally we have
\[
  dV = \left(\frac{\sqrt{-1}}{2}\right)^n d\sigma \wedge d\bar \sigma
       \wedge dz^2 \wedge d\bar z^2 \wedge \cdots \wedge dz^n \wedge d\bar z^n
  \quad \text{and} \quad
  dV_g \sim \frac{dV}{|\sigma|^2 \log^2(1/|\sigma|^2)}.
\]
Let $E$ be a holomorphic vector bundle on~$X$ with a smooth Hermitian metric~$h$. We write $\left<{\cdot},{\cdot}\right>$ for the scalar product in the fibres of $E$ induced by $h$ and $||{\cdot}||$ for the corresponding norm in the fibres of $E$.

\begin{definition}[$L^2$ spaces] \mbox{}
\begin{enumerate}
\item[(i)] Let
\[
  L^2(X, E, g) = \left\{ s \text{ measurable section of } E : \int_X ||s||^2 dV_g < \infty \right\}
\]
be the space of $L^2$ sections of $E$ with respect to the metric $g_{X'}$ with the $L^2$ norm
\[
  ||s||_{L^2(X, E, g)} = \left(\int_X ||s||^2 dV_g\right)^{1/2}.
\]
\item[(ii)] Let
\[
  L_1^2(X, E, g) = \left\{ s \in L^2(X, E, g) : \nabla s \in L^2(X, T_X^\ast \tensor E, g) \right\}
\]
be the Sobolev space of $L^2$ sections of $E$ with $L^2$ first-order weak derivatives with respect to $g_{X'}$ with the Sobolev norm
\[
  ||s||_{L_1^2(X, E, g)} = \left(||s||_{L^2(X, E, g)}^2 + ||\nabla s||_{L^2(X, T_X^\ast \tensor E, g)}^2\right)^{1/2}.
\]
Here, $\nabla$ denotes the covariant derivative with respect to the Chern connection of the Hermitian holomorphic vector bundle $(E,h)$, where $\nabla s$ is computed in the sense of currents, $T_X^\ast$ denotes the cotangent bundle of~$X$ and the bundle $T_X^\ast \tensor E$ is endowed with the product of the dual of the complete K\"ahler-Einstein metric in $T_X^\ast$ and the Hermitian metric~$h$ in~$E$.
\end{enumerate}
The spaces $L^2(X, E)$ and $L_1^2(X, E)$ are defined in the ordinary sense, i.~e.\ with respect to a smooth K\"ahler metric on~$X$.
\end{definition}

\begin{remark}
Let $\nabla = \nabla' + \nabla''$ be the decomposition of $\nabla$ into its $(1,0)$ and $(0,1)$ parts. Then for a section $s \in L^2(X, E, g)$, we have $s \in L_1^2(X, E, g)$ if and only if $\nabla's \in L^2(X, \Lambda^{1,0} T_X^\ast \tensor E, g)$ and $\nabla''s \in L^2(X, \Lambda^{0,1} T_X^\ast \tensor E, g)$. In what follows, we only consider $\nabla's$ since then everything follows for $\nabla''s$ in an analogue way.
\end{remark}

We locally write the fundamental form $\omega_{X'}$ of $g_{X'}$ as
\[
  \omega_{X'} = \sqrt{-1} \left(g_{\sigma \bar \sigma} d\sigma \wedge d\bar \sigma
              + \sum_{j=2}^n   g_{\sigma \bar \jmath} d\sigma \wedge d\bar z^j
              + \sum_{i=2}^n   g_{i \bar \sigma} dz^i \wedge d\bar \sigma
              + \sum_{i,j=2}^n g_{i \bar \jmath} dz^i \wedge d\bar z^j \right)
\]
and let
\[
  \begin{tikzpicture}[baseline=(current bounding box.center),
                      every left delimiter/.style={xshift=2ex},
                      every right delimiter/.style={xshift=-2ex}]
    \matrix (m) [matrix of math nodes,
                 nodes in empty cells,
                 left delimiter=(, right delimiter=)] {
      g^{\bar \sigma \sigma} & g^{\bar \sigma 2} & \cdots                               & g^{\bar \sigma n} \\
      g^{\bar 2 \sigma}      &                   &                                      &                   \\
      \vdots                 &                   & (g^{\bar \jmath i})_{j,i=2,\ldots,n} &                   \\
      g^{\bar n \sigma}      &                   &                                      &                   \\
    };
    
    \draw (m-4-2.south west) -- (m-2-2.north west) -- (m-2-4.north east);
  \end{tikzpicture}
  \quad \text{be the inverse matrix of} \quad
  \begin{tikzpicture}[baseline=(current bounding box.center),
                      every left delimiter/.style={xshift=2ex},
                      every right delimiter/.style={xshift=-2ex}]
    \matrix (m) [matrix of math nodes,
                 nodes in empty cells,
                 left delimiter=(, right delimiter=)] {
      g_{\sigma \bar \sigma} & g_{\sigma \bar 2} & \cdots                               & g_{\sigma \bar n} \\
      g_{2 \bar \sigma}      &                   &                                      &                   \\
      \vdots                 &                   & (g_{i \bar \jmath})_{i,j=2,\ldots,n} &                   \\
      g_{n \bar \sigma}      &                   &                                      &                   \\
    };
    
    \draw (m-4-2.south west) -- (m-2-2.north west) -- (m-2-4.north east);
  \end{tikzpicture}.
\]
Then, writing
\[
  \nabla's = s_\sigma d\sigma + \sum_{i=2}^n s_i dz^i
\]
with local sections $s_\sigma$, $s_i$ of $E$, $i = 2, \ldots, n$, the condition $\nabla's \in L^2(X, \Lambda^{1,0} T_X^\ast \tensor E, g)$ reads
\[
  \int \bigg( \left<s_\sigma, s_\sigma\right> g^{\bar \sigma \sigma}
              + \sum_{j=2}^n \left<s_\sigma, s_j\right> g^{\bar \jmath \sigma}
              + \sum_{i=2}^n \left<s_i, s_\sigma\right> g^{\bar \sigma i}
              + \sum_{i,j=2}^n \left<s_i, s_j\right> g^{\bar \jmath i}
       \bigg)
       \frac{dV}{|\sigma|^2 \log^2 (1/|\sigma|^2)} < \infty.
\]

\begin{proposition} \label{p:square-integrability}
The square-integrability conditions defined above with respect to the Poincar\'e-type metric imply the corresponding conditions in the ordinary sense, i.~e.\ we have
\begin{enumerate}
\item[(i)]  $L^2(X, E, g)   \subset L^2(X, E)$ and
\item[(ii)] $L_1^2(X, E, g) \subset L_1^2(X, E)$.
\end{enumerate}
\end{proposition}

First we need to make a remark about the asymptotic behaviour of the metric $g_{X'}$. Using Schumacher's convergence theorem and the fact that $\omega_D$ is diagonal at $p$, we see that $g^{\bar \jmath i}$ approaches $0$ for $i, j = 2, \ldots, n$ and $i \neq j$ as $\sigma \to 0$. Together with Proposition 1 from \cite{Sch02}, which is stated there in the surface case but holds analogously in higher dimensions, we obtain the following proposition which was announced above.

\begin{proposition} \label{p:asymptotics}
With $0 < \alpha \leq 1$ from Schumacher's theorem on the asymptotics of the Poincar\'e-type metric, we have

\begin{enumerate}
\item[(i)]   $g^{\bar \sigma \sigma}                    \sim |\sigma|^2 \log^2(1/|\sigma|^2)$,
\item[(ii)]  $g^{\bar \sigma i}, g^{\bar \jmath \sigma} =    O\left(|\sigma| \log^{1-\alpha}(1/|\sigma|^2)\right)$, $i, j = 2, \ldots, n$,
\item[(iii)] $g^{\bar \imath i}                         \sim 1$,                                                    $i = 2, \ldots, n$ and
\item[(iv)]  $g^{\bar \jmath i}                         \to  0$ as $\sigma \to 0$,                                  $i, j = 2, \ldots, n$, $i \neq j$.
\end{enumerate}
\end{proposition}

\begin{proof}[Proof of Proposition~\ref{p:square-integrability}]
Since the terms coming from the smooth Hermitian metric~$h$ in $E$ do not influence the following computations, we can assume that $E$ is the trivial line bundle on~$X$ and $\nabla$ is the ordinary exterior derivative $d = \partial + \dbar$.

We first observe that since $|\sigma|^2 \log^2 (1/|\sigma|^2) \to 0$ as $\sigma \to 0$, we can assume (after possibly shrinking $U$) that
\begin{equation} \label{e:poincare-denominator}
  |\sigma|^2 \log^2 (1/|\sigma|^2) \leq 1.
\end{equation}
Therefore, for every measurable function $s$, we have
\[
  \int |s|^2 \frac{dV}{|\sigma|^2 \log^2 (1/|\sigma|^2)}
  \geq \int |s|^2 dV,
\]
which implies (i). In order to show (ii), we only consider $\partial s$. Since $2 \Re(z \overline{w}) \leq |z|^2 + |w|^2$ for any complex numbers $z$ and $w$, we have
\begin{equation} \label{e:norm-ordinary}
  \begin{aligned}
    \int |\partial s|^2 dV
    &= \int \bigg( |s_\sigma|^2
                  + \sum_{j=2}^n s_\sigma \overline{s_j}
                  + \sum_{i=2}^n s_i \overline{s_\sigma}
                  + \sum_{i,j=2}^n s_i \overline{s_j}
           \bigg) dV \\
    &= \int \bigg( |s_\sigma|^2
                  + \sum_{i=2}^n |s_i|^2
                  + \sum_{j=2}^n 2 \Re (s_\sigma \overline{s_j})
                  + \sum_{\stack{i,j=2}{i<j}}^n 2 \Re (s_i \overline{s_j})
           \bigg) dV \\
    &\leq n \int \bigg( |s_\sigma|^2
                        + \sum_{i=2}^n |s_i|^2
                 \bigg) dV,
  \end{aligned}
\end{equation}
so it suffices to show that this integral is dominated by the $L^2$-norm of $\partial s$ with respect to $g_{X'}$. Now we have
\begin{equation} \label{e:two sums}
  \begin{split}
    & \quad \int \bigg( |s_\sigma|^2 g^{\bar \sigma \sigma}
                        + \sum_{j=2}^n s_\sigma \overline{s_j} g^{\bar \jmath \sigma}
                        + \sum_{i=2}^n s_i \overline{s_\sigma} g^{\bar \sigma i}
                        + \sum_{i,j=2}^n s_i \overline{s_j} g^{\bar \jmath i}
           \bigg)
           \frac{dV}{|\sigma|^2 \log^2 (1/|\sigma|^2)} \\
    & = \int \bigg( \sum_{j=2}^n \bigg( \frac{|s_\sigma|^2 g^{\bar \sigma \sigma}}{n-1}
                                        + \frac{|s_j|^2 g^{\bar \jmath j}}{n-1}
                                        + 2 \Re (s_\sigma \overline{s_j} g^{\bar \jmath \sigma})
                                 \bigg) \\
    & \quad      + \sum_{\stack{i,j=2}{i<j}}^n \bigg( \frac{|s_i|^2 g^{\bar \imath i}}{n-1}
                                                      + \frac{|s_j|^2 g^{\bar \jmath j}}{n-1}
                                                      + 2 \Re (s_i \overline{s_j} g^{\bar \jmath i})
                                               \bigg)
           \bigg)
           \frac{dV}{|\sigma|^2 \log^2 (1/|\sigma|^2)}.
  \end{split}
\end{equation}
We estimate the two sums in this expression separately. By Proposition~\ref{p:asymptotics} (i)--(iii), there are constants $c, c' > 0$ such that
\begin{align*}
  g^{\bar \sigma \sigma}   &\geq c |\sigma|^2 \log^2 (1/|\sigma|^2),       \\
  g^{\bar \jmath j}        &\geq c,                                        \\
  |g^{\bar \jmath \sigma}| &\leq c' |\sigma| \log^{1-\alpha} (1/|\sigma|^2)
\end{align*}
for $2 \leq j \leq n$. It follows that
\[
  \begin{split}
    & \quad \frac{1}{|\sigma|^2 \log^2 (1/|\sigma|^2)}
            \bigg( \frac{|s_\sigma|^2 g^{\bar \sigma \sigma}}{n-1}
                   + \frac{|s_j|^2 g^{\bar \jmath j}}{n-1}
                   + 2 \Re (s_\sigma \overline{s_j} g^{\bar \jmath \sigma})
            \bigg) \\
    & \geq  \frac{1}{(n-1) |\sigma|^2 \log^2 (1/|\sigma|^2)}
            \bigg( \! c |s_\sigma|^2 |\sigma|^2 \log^2 (1/|\sigma|^2)
                   + c |s_j|^2
                   - {2 c' (n-1) |s_\sigma| |s_j| |\sigma| \log^{1-\alpha} (1/|\sigma|^2)}
            \! \bigg) \\
    & =     \frac{c}{n-1} \bigg( |s_\sigma|^2
                                 + \left(\frac{|s_j|}{|\sigma| \log (1/|\sigma|^2)}\right)^2
                                 - \frac{2 c' (n-1) |s_\sigma| |s_j|}{c |\sigma| \log^{1+\alpha} (1/|\sigma|^2)}
                          \bigg).
  \end{split}
\]
Since $\alpha > 0$, $\log^\alpha(1/|\sigma|^2)$ tends to infinity as $\sigma$ approaches $0$. Thus we can assume (after possibly shrinking $U$) that $\log^\alpha (1/|\sigma|^2) \geq 2c'(n-1)/c$. Together with the estimate
\[
  a^2 + b^2 - ab = \frac{a^2 + b^2}{2} + \frac{(a - b)^2}{2} \geq \frac{a^2 + b^2}{2}
  \quad \text{for real numbers } a \text{ and } b
\]
and \eqref{e:poincare-denominator}, we obtain
\begin{equation} \label{e:first sum}
  \begin{split}
    & \quad \frac{1}{|\sigma|^2 \log^2 (1/|\sigma|^2)}
            \bigg( \frac{|s_\sigma|^2 g^{\bar \sigma \sigma}}{n-1}
                   + \frac{|s_j|^2 g^{\bar \jmath j}}{n-1}
                   + 2 \Re (s_\sigma \overline{s_j} g^{\bar \jmath \sigma})
            \bigg) \\
    & \geq  \frac{c}{2(n-1)} \bigg( |s_\sigma|^2
                                    + \frac{|s_j|^2}{|\sigma|^2 \log^2 (1/|\sigma|^2)}
                             \bigg) \\
    & \geq  \frac{c}{2(n-1)} \left(|s_\sigma|^2 + |s_j|^2\right).
  \end{split}
\end{equation}
The second sum in \eqref{e:two sums} can be estimated similarly to the first. Here we note that by Proposition \ref{p:asymptotics} (iv) we can assume (again after possibly shrinking $U$) that
\[
  |g^{\bar \jmath i}| \leq \frac{c}{2(n-1)}
\]
for $2 \leq i < j \leq n$. As above, it follows that
\begin{equation} \label{e:second sum}
  \begin{split}
    & \quad  \frac{1}{|\sigma|^2 \log^2 (1/|\sigma|^2)}
             \bigg( \frac{|s_i|^2 g^{\bar \imath i}}{n-1}
                    + \frac{|s_j|^2 g^{\bar \jmath j}}{n-1}
                    + 2 \Re (s_i \overline{s_j} g^{\bar \jmath i})
             \bigg) \\
    & \geq   \frac{c}{(n-1) |\sigma|^2 \log^2 (1/|\sigma|^2)}
             \bigg( |s_i|^2
                    + |s_j|^2
                    - \frac{2 (n-1) |s_i| |s_j| |g^{\bar \jmath i}|}{c}
             \bigg) \\
    & \geq   \frac{c \left(|s_i|^2 + |s_j|^2\right)}{2 (n-1) |\sigma|^2 \log^2 (1/|\sigma|^2)} \\
    & \geq   \frac{c}{2(n-1)} \left(|s_i|^2 + |s_j|^2\right).
  \end{split}
\end{equation}
Substituting \eqref{e:first sum} and \eqref{e:second sum} into \eqref{e:two sums}, we finally obtain
\[
  \begin{split}
    & \quad \int \bigg( |s_\sigma|^2 g^{\bar \sigma \sigma}
                        + \sum_{j=2}^n s_\sigma \overline{s_j} g^{\bar \jmath \sigma}
                        + \sum_{i=2}^n s_i \overline{s_\sigma} g^{\bar \sigma i}
                        + \sum_{i,j=2}^n s_i \overline{s_j} g^{\bar \jmath i}
                 \bigg)
                 \frac{dV}{|\sigma|^2 \log^2 (1/|\sigma|^2)} \\
    & \geq  \frac{c}{2(n-1)} \int \bigg( \sum_{j=2}^n \left(|s_\sigma|^2 + |s_j|^2\right)
                                         + \sum_{\stack{i,j=2}{i<j}}^n \left(|s_i|^2 + |s_j|^2\right)
                                  \bigg) dV \\
    & =     \frac{c}{2} \int \bigg( |s_\sigma|^2
                                    + \sum_{i=2}^n |s_i|^2
                             \bigg) dV,
  \end{split}
\]
which equals \eqref{e:norm-ordinary} up to a constant. This proves the claim.
\end{proof}

As explained before, Proposition~\ref{p:square-integrability} completes the proof of Theorem \ref{t:framed kobayashi-hitchin}.

%
%

%
%
\vspace{12pt}
\noindent {\scshape Fachbereich Mathematik und Informatik, Philipps-Universit\"at Marburg, Hans-Meerwein-Strasse, Lahnberge, 35032 Marburg, Germany} \\[6pt]
{\itshape E-mail address:\/} {\ttfamily stemmler@mathematik.uni-marburg.de}


\begin{thebibliography}{[38]}
\bibitem{Ba93} {\scshape S.~Bando}: {\em Einstein-Hermitian metrics on non-compact K\"ahler manifolds}, Mabuchi, Toshiki (ed.) et al., Einstein metrics and Yang-Mills connections. New York: Marcel Dekker, Inc., Lect.\ Notes Pure Appl.\ Math.\ 145, 27--33 (1993).

\bibitem{Bi97} {\scshape O.~Biquard}: {\em Fibr\'es de Higgs et connexions int\'egrables: Le cas logarithmique (diviseur lisse)}, Ann.\ Sci.\ \'Ec.\ Norm.\ Sup\'er.\ (4) 30, No.~1, 41--96 (1997).

\bibitem{Bs95} {\scshape I.~Biswas}: {\em On the cohomology of parabolic line bundles}, Math.\ Res.\ Lett.\ 2, No.~6, 783--790 (1995).

\bibitem{Bs97a} {\scshape I.~Biswas}: {\em Parabolic ample bundles}, Math.\ Ann.\ 307, No.~3, 511--529 (1997).

\bibitem{Bs97b} {\scshape I.~Biswas}: {\em Parabolic bundles as orbifold bundles}, Duke Math.\ J.\ 88, No.~2, 305--325 (1997).

\bibitem{BG96} {\scshape S.~B.~Bradlow, O.~Garc\'ia-Prada}: {\em Stable triples, equivariant bundles and dimensional reduction}, Math.\ Ann.\ 304, No.~2, 225--252 (1996).

\bibitem{BT05} {\scshape L.~Bruasse, A.~Teleman}: {\em Harder-Narasimhan filtrations and optimal destabilizing vectors in complex geometry}, Ann.\ Inst.\ Fourier 55, No.~3, 1017--1053 (2005).

\bibitem{CY80} {\scshape S.-Y.~Cheng, S.-T.~Yau}: {\em On the existence of a complete K\"ahler metric on non-compact complex manifolds and the regularity of Fefferman's equation}, Commun.\ Pure Appl.\ Math.\ 33, 507--544 (1980).

\bibitem{Do83} {\scshape S.~K.~Donaldson}: {\em A new proof of a theorem of Narasimhan and Seshadri}, J.\ Differ.\ Geom.\ 18, 269--277 (1983).

\bibitem{Do85} {\scshape S.~K.~Donaldson}: {\em Anti self-dual Yang Mills connections over complex algebraic surfaces and stable vector bundles}, Proc.\ Lond.\ Math.\ Soc., III.\ Ser.\ 50, 1--26 (1985).

\bibitem{Do87} {\scshape S.~K.~Donaldson}: {\em Infinite determinants, stable bundles and curvature}, Duke Math.\ J.\ 54, 231--247 (1987).

\bibitem{Ga54} {\scshape M.~P.~Gaffney}: {\em A special Stokes's theorem for complete Riemannian manifolds}, Ann.\ Math.\ (2) 60, 140--145 (1954).

\bibitem{GP94} {\scshape O.~Garc\'ia-Prada}: {\em Dimensional reduction of stable bundles, vortices and stable pairs}, Int.\ J.\ Math.\ 5, No.~1, 1--52 (1994).

\bibitem{Hoe90} {\scshape L.~H\"ormander}: {\em An introduction to complex analysis in several variables. 3rd revised ed.}, North-Holland Mathematical Library, 7. Amsterdam etc.: North Holland (1990).

\bibitem{Ko84} {\scshape R.~Kobayashi}: {\em K\"ahler-Einstein metric on an open algebraic manifold}, Osaka J.\ Math.\ 21, 399--418 (1984).

\bibitem{Kb80} {\scshape S.~Kobayashi}: {\em First Chern class and holomorphic tensor fields}, Nagoya Math.\ J.\ 77, 5--11 (1980).

\bibitem{Kb82} {\scshape S.~Kobayashi}: {\em Curvature and stability of vector bundles}, Proc.\ Japan Acad., Ser.\ A	58, 158--162 (1982).

\bibitem{Le93} {\scshape M.~Lehn}: {\em Moduli spaces of framed vector bundles (Modulr\"aume gerahmter Vektorb\"undel)}, Bonner Mathematische Schriften. 241. Bonn: Univ.\ Bonn (1993).

\bibitem{LN99} {\scshape J.~Li, M.~S.~Narasimhan}: {\em Hermitian-Einstein metrics on parabolic stable bundles}, Acta Math.\ Sin., Engl.\ Ser.\ 15, No.~1, 93--114 (1999).

\bibitem{Lue83} {\scshape M.~L\"ubke}: {\em Stability of Einstein-Hermitian vector bundles}, Manuscr.\ Math.\ 42, 245--257 (1983).

\bibitem{Lue93} {\scshape M.~L\"ubke}: {\em The analytic moduli space of framed vector bundles}, J.\ Reine Angew.\ Math.\ 441, 45--59 (1993).

\bibitem{LOS93} {\scshape M.~L\"ubke, C.~Okonek, G.~Schumacher}: {\em On a relative Kobayashi-Hitchin correspondence}, Int.\ J.\ Math.\ 4, No.~2, 253--288 (1993).

\bibitem{MY92} {\scshape M.~Maruyama, K.~Yokogawa}: {\em Moduli of parabolic stable sheaves}, Math.\ Ann.\ 293, No.~1, 77--99 (1992).

\bibitem{MS80} {\scshape V.~B.~Mehta, C.~S.~Seshadri}: {\em Moduli of vector bundles on curves with parabolic structures}, Math.\ Ann.\ 248, 205--239 (1980).

\bibitem{NS65} {\scshape M.~S.~Narasimhan, C.~S.~Seshadri}: {\em Stable and unitary vector bundles on a compact Riemann surface}, Ann.\ Math.\ (2) 82, 540--567 (1965).

\bibitem{Ni02} {\scshape L.~Ni}: {\em The Poisson equation and Hermitian-Einstein metrics on holomorphic vector bundles over complete noncompact K\"ahler manifolds}, Indiana Univ.\ Math.\ J.\ 51, No.~3, 679--704 (2002).

\bibitem{NR01} {\scshape L.~Ni, H.~Ren}: {\em Hermitian-Einstein metrics for vector bundles on complete K\"ahler manifolds}, Trans.\ Am.\ Math.\ Soc.\ 353, No.~2, 441--456 (2001).

\bibitem{Po05} {\scshape D.~Popovici}: {\em A simple proof of a theorem by Uhlenbeck and Yau}, Math.\ Z.\ 250, No.~4, 855--872 (2005).

\bibitem{Sc04} {\scshape A.~Schmitt}: {\em A universal construction for moduli spaces of decorated vector bundles over curves}, Transform.\ Groups 9, No.~2, 167--209 (2004).

\bibitem{Sch98} {\scshape G.~Schumacher}: {\em Asymptotics of K\"ahler-Einstein metrics on quasi-projective manifolds and an extension theorem on holomorphic maps}, Math.\ Ann.\ 311, No.~4, 631--645 (1998).

\bibitem{Sch02} {\scshape G.~Schumacher}: {\em Asymptotics of complete K\"ahler-Einstein metrics --- negativity of the holomorphic sectional curvature}, Doc.\ Math., J.\ DMV 7, 653--658 (2002).

\bibitem{Si88} {\scshape C.~T.~Simpson}: {\em Constructing variations of Hodge structure using Yang-Mills theory and applications to uniformization}, J.\ Am.\ Math.\ Soc.\ 1, No.~4, 867--918 (1988).

\bibitem{St09} {\scshape M.~Stemmler}: {\em Stability and Hermitian-Einstein metrics for vector bundles on framed manifolds}, Dissertation, Philipps-Universit\"at Marburg (2009).

\bibitem{Ta72} {\scshape F.~Takemoto}: {\em Stable vector bundles on algebraic surfaces}, Nagoya Math.\ J.\ 47, 29--48 (1972).

\bibitem{TY87} {\scshape G.~Tian, S.~T.~Yau}: {\em Existence of K\"ahler-Einstein metrics on complete K\"ahler manifolds and their applications to algebraic geometry}, Mathematical aspects of string theory, Proc.\ Conf., San Diego/Calif.\ 1986, Adv.\ Ser.\ Math.\ Phys.\ 1, 574--629 (1987).

\bibitem{UY86} {\scshape K.~Uhlenbeck, S.~T.~Yau}: {\em On the existence of Hermitian-Yang-Mills connections in stable vector bundles}, Commun.\ Pure Appl.\ Math.\ 39, 257--293 (1986).

\bibitem{UY89} {\scshape K.~Uhlenbeck, S.~T.~Yau}: {\em A note on our previous paper: On the existence of Hermitian Yang-Mills connections in stable vector bundles}, Commun.\ Pure Appl.\ Math.\ 42, No.~5, 703--707 (1989).

\bibitem{Yau78} {\scshape S.-T.~Yau}: {\em On the Ricci curvature of a compact K\"ahler manifold and the complex Monge-Amp\`ere equation. I.}, Commun.\ Pure Appl.\ Math.\ 31, 339--411 (1978).
\end{thebibliography}
\end{document}